\documentclass{amsart}

\usepackage{amssymb}
\usepackage{amsmath}
\usepackage{amsfonts}
\usepackage{amsthm}
\usepackage{epsfig}
\usepackage{amscd}
\usepackage{stmaryrd}
\usepackage{hyperref}
\usepackage[all]{xy}
   \SelectTips{cm}{10}
\renewcommand{\ell}{l}

\newtheorem{theorem}{Theorem}[section]
\theoremstyle{definition}
\newtheorem{prop}[theorem]{Proposition}
\newtheorem{lemma}[theorem]{Lemma}
\newtheorem{sublemma}[theorem]{Sublemma}

\newtheorem{df}[theorem]{Definition}
\newtheorem{remark}[theorem]{Remark}
\newtheorem{corr}[theorem]{Corollary}

\def\wpi{\widetilde{\pi}}
\def\sc{\mathrm{sc}}
\def\C{\mathbf{C}}

\def\h{\mathfrak{h}}
\def\t{\mathfrak{t}}
\def\g{\mathfrak{g}}
\def\e{\mathfrak{e}}
\def\f{\mathfrak{f}}
\def\sl{\mathfrak{sl}}
\def\so{\mathfrak{so}}
\def\sp{\mathfrak{sp}}

\def\mr{\mathrm}
\def\tv{\tilde{v}}

\def\Out{\mathrm{Out}}
\def\rank{\mathrm{rank}}
\def\SL{\mathrm{SL}}

\def\GL{\mathrm{GL}}
\def\Fbar{\overline{\F}}
\def\F{\mathbf{F}}
\def\Qbar{\overline{\Q}}
\def\Q{\mathbf{Q}}
\def\Z{\mathbf{Z}}
\def\Gal{\mathrm{Gal}}

\def\R{\mathbf{R}}
\def\Zbar{\overline{\Z}}

\def\min{\mathrm{min}}

\def\rbar{\overline{r}}
\def\rhobar{\overline{\rho}}
\def\pr{\mathrm{pr}}

\def\TT{S_{\mathrm{reg}}}
\def\elll{l}
\def\kapppa{\epsilon}
\def\Gammma{{G}}
\def\ww{\widetilde{w}}

\thanks{G.B. was supported in
part by NSF postdoctoral fellowship DMS-1503047,
F.C. was supported in part by NSF Grant
  DMS-1701703, M.E. was supported in part by NSF
  Grant DMS-1601871,  K.M.P. was supported in part by NSF
  Grant DMS-1803623, S.P. was supported in part by NSF
  Grant DMS-1700759.}

\begin{document}

\title{Compatible Systems of Galois Representations Associated To The Exceptional Group~$E_6$}

\author[G.~Boxer]{George Boxer}
\address{Department of Mathematics, University of Chicago, 5734 S University Ave
Chicago, IL 60637}
\email{gboxer@math.uchicago.edu}
\author[F.~Calegari]{Frank Calegari}
\address{Department of Mathematics, University of Chicago, 5734 S University Ave
Chicago, IL 60637}
\email{fcale@math.uchicago.edu}
\author[M.~Emerton]{Matthew Emerton}
\address{Department of Mathematics, University of Chicago, 5734 S University Ave
Chicago, IL 60637}
\email{emerton@math.uchicago.edu}
\author[B.~Levin]{Brandon Levin}
\address{Department of Mathematics,
University of Arizona, 
617 N. Santa Rita Avenue, 
Tucson, Arizona 85721}
\email{bwlevin@math.arizona.edu}
\author[K.~Madapusi Pera]{Keerthi  Madapusi Pera}
\address{Department of Mathematics, Boston College,
Chestnut Hill, MA 02467}
\email{keerthi.madapusipera@bc.edu}
\author[S.~Patrikis]{Stefan Patrikis}
\address{Department of Mathematics, The University of Utah, 155 S 1400 E, Salt Lake City, UT 84112}
\email{patrikis@math.utah.edu}
\maketitle

\section{Introduction}

In~\cite{SerreMotives}, Serre raised the question of whether~$G_2$ or~$E_8$ was the 
motivic Galois group of a motive~$M$ over a number field, and one can evidently ask the same question for the other exceptional simple Lie groups. A slightly weaker 
version of this question asks for a motive~$M$ such that the associated~$p$-adic Galois
representations have algebraic monodromy group  equal to the exceptional group in question.
In this form, Serre's question was answered in the affirmative by Yun in~\cite{Yun}, who also dealt with
the case of the exceptional group~$E_7$. (A stronger version of the question for the group~$G_2$
had previously been answered by Dettweiler and Reiter~\cite{DR}.) This left open the cases of~$E_6$ or~$F_4$. In~\cite{Patrikis}, the last author of this paper succeeded in constructing
geometric Galois representations for the remaining exceptional groups using arguments inspired
by Ramakrishna's lifting theorems~\cite{Ramakrishna}, at least for a set of primes~$p$ of density one (improved
to all but finitely many in Theorem~1.2 of~\cite{Patrikis2}). While this answered a weak form
of (the~$E_6$ analogue of) Serre's question, the Galois
representations constructed in~\cite{Patrikis,Patrikis2} did not obviously come
from motives~$M$ or from compatible systems of Galois representations
(although that would certainly be a consequence of the Fontaine--Mazur conjectures~\cite{FM}).
The main goal of this paper is to remedy this lacuna for the group~$E_6$.

\begin{theorem} \label{theorem:main}
Let~$F/F^{+}$ be a totally imaginary CM field with maximal totally real subfield~$F^{+}$. Let~$G$  denote the 
simply connected form of~$E_6$, and fix a minuscule representation~$G \rightarrow \GL_{27}$. 
Then there exists a strongly compatible system of Galois representations with coefficients in a number field~$M$ such that
the representations
$$r_{\lambda}: G_{F} \rightarrow G(\overline{M}_{\lambda}) \hookrightarrow \GL_{27}(\overline{M}_{\lambda})$$
have images with  Zariski closure~$G(\overline{M}_{\lambda})$ for all primes~$\lambda$.
Moreover, this compatible system is potentially automorphic and motivic in the sense that there is a CM extension $H/F$ such that:
\begin{itemize}
\item There a cuspidal automorphic representation $\pi$ for $\GL_{27}/H$ such that $r_\lambda|_{G_H}$ is the compatible system of Galois representations associated to $\pi$.
\item The compatible system $r_{\lambda}$ satisfies the conclusion of the Fontaine--Mazur conjecture: there is a smooth projective variety $X/F$ and integers $i$ and $j$ such that $r_{\lambda}$ is a $\Gammma_F$-sub-representation of $H^i(X_{\overline{F}}, \Qbar_{\ell}(j))$. 

\end{itemize}

\end{theorem}

The main idea of the paper is to follow the strategy of~\cite{Patrikis}, but to replace  the  lifting theorems inspired by Ramakrishna with those
inspired by the work of Khare and Wintenberger~\cite{KWI,KWII},  exploiting the modularity lifting 
theorems of~\cite{BLGGT}. In order to do this, one must link~$G$-representations with~$\GL_n$
representations by choosing some (faithful) representation~$r \colon G \to \GL_{n}$. The methods
of~\cite{BLGGT} require that the corresponding Galois representations have \emph{distinct}
Hodge--Tate weights. This imposes a strong restriction on the representation~$r$, namely, that
its formal character should be multiplicity-free.  
In particular, the method of this paper only applies to the exceptional groups~$G_2$,
$E_6$, and~$E_7$, in their quasi-minuscule ($G_2$) or minuscule ($E_6$ and $E_7$) representations (we concentrate on~$E_6$ because other methods are available in the other cases).
As in~\cite{Patrikis,Patrikis2}, we require a seed representation~$\rhobar: G_{F} \rightarrow G(\Fbar_p)$
from which to construct geometric lifts. In~\cite{Patrikis}, suitable representations~$\rhobar$
came from composing representations associated to modular forms with the principal~$\SL_2$. These
representations are not suitable for our purposes, because their composition with the minuscule
representation is reducible. Instead, we construct a representation  related  to the action
of the Weyl group of~$E_6$ on the weight space of our representation. The fact that the
 representation we consider is irreducible in~$\GL_n$  relies on the assumption that~$r$
is minuscule.
The reason we succeed in controlling the monodromy groups at \emph{all} primes is a consequence
of elementary combinatorial properties of the formal character of~$E_6$ (using ideas of Larsen and Pink~\cite{LP}) together 
with our ability to exploit 
independence of~$p$ results in compatible systems of Galois representations associated to 
automorphic forms (\cite{Yoshida, MR2800722, Caraiani}).

We end the introduction with some remarks on what the methods of this paper cannot do.

\begin{remark}{\bf Galois representations for $\Q$ versus imaginary quadratic fields.} Our construction gives compatible systems of~$E_6$-representations over any imaginary quadratic field~$F$; these extend to representations of $G_{\Q}$ whose image is Zariski-dense in the L-group $G \rtimes \Out(G)$ of an outer form of $E_6$.
We leave open the question as to whether actual $E_6$-systems exist over~$\Q$, noting that the methods of
this paper will not succeed in constructing them. Indeed, our methods require that the corresponding
Galois representations have regular weight, and there do not exist any such compatible systems of Galois
representations over~$\Q$ (see Remark~\ref{remark:irregular}).
\end{remark}

\begin{remark} {\bf Motives versus motives with coefficients\rm}. We ultimately
construct compatible systems  of~$27$-dimensional
Galois representations for a coefficient field~$L$ over which we have little control. One can ask the more refined question
of whether there exists a motive~$M$ with coefficients over~$\Q$ of type~$E_6$ (Yun's result~\cite{Yun} 
answer this question in the affirmative for~$E_8$).  One reason that this refinement is interesting is that it would have applications to the to the inverse Galois problem (see Corollary \ref{cor}).  

 It seems to us that the Galois theoretic methods
of either this paper or of~\cite{Patrikis,Patrikis2} are unsuited to answering such a question. It is illustrative
to consider the simpler case of~$\GL_2$. By constructing a geometric Galois representation~$\rho: G_{\Q} \rightarrow \GL_2(\Q_p)$, one can hope to prove it is automorphic and hence associated to a modular form~$f$,  and thus to
construct  a  corresponding
motive~$M_f$ (\cite{MR1047142}). But it seems very hard to impose conditions on~$\rho$ to ensure that the form~$f$ has coefficients in~$\Q$. For example, in weight~$2$ (on the modular form side) one would want to put conditions on~$\rho$ to ensure that it actually come from an elliptic curve rather than an abelian variety of~$\GL_2$-type. In practice, we actually work in highly regular
weight, and there is a certain amount of numerical evidence~\cite{Dave} pointing to the fact that there may not exist any 
motives~$M$ at all with coefficients in~$\Q$ and monodromy group~$\GL_2$ with 
Hodge--Tate weights~$[0,k-1]$ when~$k>50$.
\end{remark}

\section{The mod $p$ representation}

In this section, we construct the mod $p$ representations that we will lift in the next section using potential automorphy theorems.

\begin{df}
Let~$G$ be the split simply-connected reductive group scheme over $\Z$ of type $E_6$. Fix a pinned based root datum of $G$, and let ${}^{L}G= G \rtimes \Out(G)$ with the non-trivial element $\tau$ of $\Out(G)= \Z/2\Z$ acting through the corresponding pinned automorphism. The induced action of $\tau$ on $W_G$ is conjugation by the longest element $w_0$, since $\tau$ acts on $T$, the maximal torus of the pinning, by the opposition involution $-w_0$. Let~$W = W_G \rtimes \Out(G)$, and continue to write $\tau$ for the non-trivial element of $\Out(G) \subset W$. We will write $B$ for the Borel subgroup of $G$ associated to the based root datum.
\end{df}

\begin{remark}\label{outerweyl} \emph{There is an isomorphism~$W \simeq W_G \times \Z/2\Z$ given by the identity on $W_G$ and sending~$\tau$ to~$(w_0,1)$.}
\end{remark}

We fix, once and for all, a choice of minuscule representation~$r_{\min} \colon G \to \GL_{27}$, writing $\Lambda_{\mr{min}}$ for the weights of $T$ in $r_{\mr{min}}$. Now we explain how to extend $r_{\min}$ to ${}^{L}G$. Let $\mathcal{G}_{27}= (\GL_{27} \times \GL_1) \rtimes \Z/2\Z$, where the non-trivial element $\jmath \in \Z/2\Z$ acts via $\jmath  \kern+0.1em{(g, \mu)}\jmath^{-1}= (\mu \cdot {}^t g^{-1}, \mu)$. The representation $g \mapsto r_{\min}(\tau g \tau^{-1})$ is isomorphic to the dual minuscule representation of $G$, so there exists $A \in \GL_{27}$ such that $r_{\min}(\tau g \tau^{-1})= A \cdot {}^t r_{\min}(g)^{-1} A^{-1}$. Iterating, we see that $A \cdot {}^t A^{-1}$ commutes with $r_{\min}$, so must be a scalar: $A= {}^t A \cdot \varepsilon$. Clearly $\varepsilon \in \{\pm 1\}$, and since 27 is odd, we must in fact have $\varepsilon =1$ by considering determinants. 

We can now extend $r_{\min}$ to 
\[
r_{\min} \colon {}^{L}G \to \mathcal{G}_{27}
\]
by $r_{\min}(\tau)= (A, \varepsilon, \jmath)= (A, 1, \jmath)$. The fact that $\varepsilon=1$ has the following consequence, which we will need later:
\begin{lemma} 
Let $\nu \colon \mathcal{G}_{27} \to \mathbf{G}_m$ be the character given by $\nu(g, a, 0)= a$, $\nu(\jmath)=-1$. Let $x \in {}^{L}G$ be any element with non-trivial projection to $\Out(G)$ (the case of interest will be $x$ such that conjugation by $x$ induces a split Cartan involution of $G$). Then $\nu \circ r_{\min}(x)=-1$.
\end{lemma}

\begin{lemma}\label{Lconj} Let $E^+/F^+$ be a Galois extension of totally real fields whose Galois group is identified with a subgroup $P$ of $W_G$, and let $F/F^+$ be a quadratic totally imaginary extension. Let $E=F.E^+$.
Then the composite
\[
 \Gal(E/F^+) \xrightarrow{\sim} \Gal(E^+/F^+) \times \Gal(F/F^+) \xrightarrow{\sim} P \times \Z/2\Z \subset W_G \times \Z/2\Z \xrightarrow{\sim} W
\]
sends complex conjugation to $(w_0, \tau) \in W= W_G \rtimes \Out(G)$. In this setting we will write ${}^L P$ for the image of $P \times \Z/2\Z$ in $W$.
\end{lemma}

\begin{proof} This follows from the definition of the isomorphism $W_G \times \Z/2\Z \xrightarrow{\sim} W$ (Remark \ref{outerweyl}).
\end{proof}

We now construct the mod $p$ Galois representation. Let $P$ be a Sylow 3-subgroup of $W_G$; we note that $|W_G|= 2^7 \cdot 3^4 \cdot 5$. 
Let~$F/F^+$ denote a fixed totally imaginary quadratic extension of our fixed totally real field $F^+$. Our starting point will be to construct $P$ as a Galois group over $F^+$, with some local restrictions whose significance will become apparent when we apply potential automorphy theorems.

\begin{lemma} There exists a totally real Galois extension~$E^{+}/F^{+}$  together with an equality~$\Gal(E^{+}/F^{+}) = P$ (we write ``$=$'' to denote a fixed isomorphism)
and all the primes of $F^+$ which ramify in~$E^{+}$ are split in~$F/F^{+}$.
\end{lemma}

\begin{proof} The Scholz--Reichardt theorem~(\cite[Thm 2.1.1]{SerreNotes}) guarantees the existence of a number field~$L$ with~$\Gal(L/\Q) = P$, and, because~$|P|$ is odd,
such an extension will automatically be totally real. It suffices to show that we can construct such an extension ramified only at primes which are
totally split in~$F$ (equivalently, in the Galois closure of~$F$).
This forces the intersection of~$L$ with the Galois closure of~$F$ to be unramified everywhere over~$\Q$ and hence trivial, and thus~$E^{+} = L.F^{+}$
will have Galois group~$P$ over~$F$ and produce the desired extension.
The result follows immediately by induction and from the following lemma,
which is extremely close to~\cite[Thm~2.1.3]{SerreNotes}:
\begin{sublemma}Let~$\widetilde{A} \rightarrow A$ be a central extension of a finite group~$A$ by~$\Z/p \Z$, and assume that the exponent of~$\widetilde{A}$ divides~$p^n$.
Let~$L/\Q$ be Galois with Galois group~$A$, and assume that every prime~$l$ which ramifies in~$L$ has the following properties:
\begin{enumerate}
\item $l \equiv 1 \mod p^n$.
\item \label{case:two}   The inertia group(s) of~$l$ in~$A$ coincides with the decomposition group(s) at~$l$.
\item $l$ splits completely in~$F$.
\end{enumerate}
Then there exists an extension~$\widetilde{L}/L$ which is Galois over~$\Q$ with Galois group $\widetilde{A}$ and such that the primes~$l$ which ramify in~$\widetilde{L}$ satisfy the same conditions as above.
\end{sublemma}
The only difference between
this statement and Theorem~2.1.3 of~\cite{SerreNotes} is the extra requirement that the primes~$l$
splits completely in~$F$. There are two inductive steps in the proof of Theorem~2.1.3 of~\cite{SerreNotes}, and we indicate the required argument to show that the new auxiliary prime~$q$
may be chosen to split completely in~$F$.

Suppose first that~$\widetilde{A}=A\times\Z/p\Z$ is a split extension. Pick any prime~$q \equiv 1 \mod p^n$ which is totally split in the Galois closure of $F$ and  also totally split
 in the field~$L(\zeta_{p^n},\{l^{1/p}\}_{l \in \mathrm{Ram}(L/\Q)})$ given
by adjoining the~$p$th roots of primes~$l$ which ramify in~$L$. Then take $\widetilde{L}$ to be the composite of $L$ and the sub-extension of $\Gal(\Q(\zeta_q)/\Q)$ with Galois group $\Z/p\Z$.

Now suppose that~$\widetilde{A}$ is a non-split extension. The argument in~\cite{SerreNotes} proceeds by first finding an extension~$\widetilde{L}$ and then modifying~$\widetilde{L}$
so that it is ramified at the same places as~$L$. Hence the ramified primes~$l$ automatically satisfy the required splitting condition in~$F$.
 The final step is to modify the field further so that it has property~(\ref{case:two}).  This is achieved by choosing an auxiliary prime~$q \equiv 1 \mod p$ along with a character~$\chi: (\Z/q\Z)^{\times} \rightarrow \Z/p \Z$ satisfying the following properties:
\begin{enumerate}
\item The prime~$q \equiv 1 \mod p^n$.
\item For every prime~$l$ which ramifies in~$L$,  there is an equality $\chi(l) = c_l$ where~$c_l$ is determined from the extension~$L$.
\item The prime~$q$ splits completely in~$L$.
\item The prime~$q$ splits completely in~$F$.
\end{enumerate}
Only the last condition is new.
The first three conditions are \v{C}ebotarev conditions in the field~$L(\zeta_{p^n}, \{l^{1/p}\}_{l \in \mathrm{Ram}(L/\Q)})$, whereas the fourth condition is a \v{C}ebotarev condition in the
Galois closure of~$F$. By construction, the primes~$l$ are totally split in the Galois closure of~$F$. Hence the intersection of the
Galois closure of~$F$ with~$L(\zeta_{p^n}, \{l^{1/p}\}_{l})$ must be contained inside~$\Q(\zeta_{p^n})$. Since the first  condition implies that~$q$
splits completely in~$\Q(\zeta_{p^n})$, there is no obstruction to finding such primes~$q$  satisfying all
four conditions using the \v{C}ebotarev density theorem provided there is no obstruction without the last hypothesis.
But this is exactly what follows from the proof of Theorem~2.1.3 of~\cite{SerreNotes}.
\end{proof}

As in Lemma \ref{Lconj}, let ${}^L P$ be the image of $P \times \Z/2\Z$ in $W$; explicitly, it is the product $P \times \langle (w_0, \tau)\rangle$ inside $W$. Our reason for working with the group $P$ is the combination of the following two properties:

\begin{lemma}  \leavevmode \label{lemma:split}
\begin{enumerate}
\item The restriction of the extension
\[
 1 \to T(\Z) \to N_G(T)(\Z) \rtimes \Out(G) \xrightarrow{\pi} W_G \rtimes \Out(G) \to 1
\]
to ${}^L P$ splits.
\item  \label{splittransitive} $P$ acts transitively on the set $\Lambda_{\mr{min}}$ of weights of the minuscule representation $r_{\mr{min}}$.
\end{enumerate}
\end{lemma}

\begin{proof}
Since $P$ is a 3-group while $T(\Z)$ is an $\F_2$-module, the Hochschild--Serre spectral sequence 
\[
H^i(P, H^j(\langle (w_0, \tau)\rangle, T(\Z))) \implies H^{i+j}({}^L P, T(\Z))
\]
degenerates at the $E_2$-page, and restriction induces an isomorphism 
$$H^2({}^L P, T(\Z)) \xrightarrow{\sim} H^2(\langle (w_0, \tau)\rangle, T(\Z))^P.$$
 The image of our extension under this restriction isomorphism is trivial, because the element $(w_0, \tau)$ lifts to an order two element of $N(T)(\Z) \rtimes \Out(G)$ (see for instance \cite[Lemma 3.1]{adams-he}, noting that $Z_G$ has order 3).

For the second part of the lemma, note that $W_G$ acts transitively on $\Lambda_{\mr{min}}$, so the 3-part of the stabilizer of any element $\lambda \in \Lambda_{\mr{min}}$ has order 3 (recall $|W_G|= 2^7\cdot 3^4\cdot 5$ and $|\Lambda_{\mr{min}}|= 3^3$). It follows that the orbit of $P$ on $\lambda$ has order at least $|P|/3=3^3$, and thus (equality holds and) $P$ acts transitively. (More generally, a finite group~$G$ acts transitively
on a set~$X$ of~$p$-power order if and only if a~$p$-Sylow subgroup~$P$ acts transitively on~$X$.)
\end{proof} 

The work of Shafarevich on the inverse Galois problem for solvable groups implies that every split embedding problem with nilpotent kernel has a proper solution; we need some precise local control so will not be able to invoke this theorem, and unfortunately the following construction, guided by the demands of the local deformation theory as in \cite{Patrikis} and of automorphy lifting as in \cite{BLGGT}, is somewhat technical. 

Applying Lemma~\ref{lemma:split}, let us fix a splitting of the extension
\[
 1 \to T(\Z) \to \pi^{-1}({}^L P) \xrightarrow{\pi} {}^L P \to 1
\]
and write $s \colon \Gal(E/F^+) \to \pi^{-1}({}^L P)$ for the resulting lift. This representation is not yet suitable for potential automorphy theorems, so we modify it in the following Proposition. 

We first establish some notation. For a fixed prime $p$, we write $\kapppa$ for the $p$-adic cyclotomic character and $\bar{\kapppa}$ for its mod $p$ reduction. For a prime ${\elll}$ exactly dividing $p-1$, let $\pr({\elll})$ be the projection from ${\F}_p^{\times}$ onto the ${\elll}$-torsion subgroup ${\F}_p^{\times}[{\elll}]$ restricting to the identity on ${\F}_p^{\times}[{\elll}]$. Finally, let $\bar{\kapppa}[{\elll}]= \pr({\elll}) \circ \bar{\kapppa} \colon \Gammma_{F^+} \to {\F}_p^{\times}[{\elll}]$. For any homomorphism $\bar{\rho} \colon G_{F^+} \to G(\F_p)$ and any place $v$ of $F^+$, let $\bar{\rho}_{v} := \bar{\rho}|_{G_{F^+_v}}$. Finally, for any homomorphism of groups $\rho \colon \Gamma \to \Gamma'$, and any $\rho(\Gamma)$-module $M$, let $\rho(M)$ denote $M$ regarded as a $\Gamma$-module.  

\begin{prop}\label{seed}
Consider pairs of primes~$(\elll,p)$ such that:
\begin{itemize}
\item All primes above~$\elll$ split in~$F/F^{+}$.
\item $p$ splits in~$E/\Q$ and~$p-1$ is divisible by~$\elll$ but not by~$\elll^2$.
\end{itemize}
Then there exist infinitely many primes~$\elll$ such that there exist infinitely many  pairs~$(\elll,p)$
  such that there exists a homomorphism
\[
 \bar{\rho} \colon \Gammma_{F^+} \to T({\F}_p)[{\elll}] \cdot \pi^{-1}({}^L P) \subset N_G(T)({\F}_p) \rtimes \Out(G)
\]
lifting our fixed identification $\Gal(E/F^+)= {}^L P$ and satisfying the following:
\begin{enumerate}
 \item The restriction $\rhobar|_{\Gammma_{F}}$ is ramified only at places split in $F/F^+$. 
 \item For any choice of complex conjugation $c \in \Gammma_{F^+}$, $\rhobar(c)$ is a split Cartan involution of $G$, i.e. $\mathrm{dim}(\mathfrak{g}^{\mathrm{Ad}(\rhobar(c))=1})= \dim(G/B) = 36$. 
 \item For all places $v \vert p$ of $F^+$, fix any choice of integers $n_{v, \alpha}$ indexed by simple roots $\alpha \in \Delta= \Delta(G, B, T)$. Then $\rhobar|_{\Gammma_{F^+_v}}$ is equal to 
\[ 
\prod_{\alpha \in \Delta} \alpha^\vee \circ \bar{\kapppa}[{\elll}]^{n_{v, \alpha}}.  
\]
 \item Let $\rho_G^\vee$ denote the half-sum of the positive coroots of $(G, B, T)$; it lies in $X_{\bullet}(T)$ since $\#Z_G=3$. There is a set~$\TT$ 
  of two primes $q$ split in $E/\Q$ and of order ${\elll}$ modulo $p$ such that for some place $v \vert q$ of $F^+$, $\rhobar|_{\Gammma_{F^+_v}}$ is unramified with Frobenius mapping to $\rho_G^\vee(q)$ (which  lands
  inside the group~$T(\F_p)[l]$ because $q^l \equiv 1 \pmod p$). 
\end{enumerate}
Moreover, in addition to satisfying the above conditions, we can choose $p> 56= 2\cdot(27+1)$, ${\elll} > h_G = 12$ (the Coxeter number of $G$), 
 $\{n_{v, \alpha}\}$, and $\rhobar$ such that
\begin{enumerate}
\item For all $v \vert p$, the composite $r_{\mathrm{min}} \circ \rhobar|_{\Gammma_{F^+_v}}$ is a direct sum of distinct powers of $\bar{\kapppa}[{\elll}]$. 
\item For all $v \vert p$ and for any Borel subgroup $B$ (with Lie algebra $\mathfrak{b}$) containing $T$, the cohomology groups $H^0(\Gammma_{F^+_v}, \rhobar_v(\mathfrak{g}/\mathfrak{b}))$ and $H^0(\Gammma_{F^+_v}, \rhobar_v(\mathfrak{g}/\mathfrak{b})(1))$ both vanish. Moreover, for any Borel subgroup $B_{27}$ containing the maximal torus of $\mathrm{SL}_{27}$ that stabilizes the weight spaces of $r_{\mathrm{min}}$, $H^0(\Gammma_{F^+_v}, (r_{\mathrm{min}}\circ\rhobar_v)(\mathfrak{sl}_{27}/\mathfrak{b}_{27}))=0$ and $H^0(\Gammma_{F^+_v}, (r_{\mathrm{min}}\circ\rhobar_v)(\mathfrak{sl}_{27}/\mathfrak{b}_{27})(1))=0$.
\item The composite $r_{\mathrm{min}}\circ \rhobar|_{\Gammma_{F(\zeta_p)}}$ is absolutely irreducible.
\end{enumerate}
\end{prop}
\begin{proof}
 For now let ${\elll}$ be any odd prime such that the places of $F^+$ above ${\elll}$ are all split in $F/F^+$, and $E$ and $\Q(\mu_{{\elll}^2})$ are linearly disjoint over $\Q$ (eg, take ${\elll}$ split in $E/\Q$); later in the argument we will require ${\elll}$ to be larger than some absolute bound depending only on the group $G$. Let $p$ be any prime split in $E/\Q$ such that ${\elll}$, but not ${\elll}^2$, divides $p-1$ (such $p$ exist by 
 \v{C}ebotarev). 
 We will modify the original lift $s \colon \Gal(E/F^+) \to \pi^{-1}({}^L P)$ by an element of $H^1(\Gammma_{F^+}, T({\F}_p[{\elll}]))$ so as to satisfy the local conditions (here $\Gammma_{F^+}$ continues to act via the quotient $\Gal(E/F^+)$; note that $T({\F}_p)[{\elll}]$ is $\tau$-stable). Let $T({\F}_p)[{\elll}]= \oplus W_i$ be the decomposition into irreducible ${\F}_{{\elll}}[P]$-modules. Recall that $(w_0, c) \in {}^L P$ acts on $T$ by $-1$, so this is also a decomposition as ${}^L P$-module, and the splitting field $F^+(W_i)$ of the $\Gal(E/F^+)$-module $W_i$ contains $F$ for each $i$ (since ${\elll} \neq 2$). Let $\Sigma$ be the set of places of $F^+$ that are either split or ramified in $F/F^+$ (implicitly including the infinite places in the former condition); in particular, the action of $\Gammma_{F^+}$ on $T({\F}_p)[{\elll}]$ factors through $\Gammma_{F^+, \Sigma}$. For any finite subset $T$ of $\Sigma$, \cite[Theorem 9.2.3(v)]{Neukirch} implies that the restriction map
\[
 H^1(\Gammma_{F^+, \Sigma}, W_i) \to \bigoplus_T H^1(\Gammma_{F^+_v}, W_i)
\]
is surjective for all $i$. Assembling the different $i$ (with a common set $T$), the restriction map 
\[
 H^1(\Gammma_{F^+, \Sigma}, T({\F}_p)[{\elll}]) \to \bigoplus_T H^1(\Gammma_{F^+_v}, T({\F}_p)[{\elll}])
\]
is also surjective. We apply this observation to the following set $T$ and the following local cohomology classes: let $T$ be the union of the  following sets of places of $F^+$:
\begin{itemize}
 \item places which are ramified in $F/F^+$;
 \item places dividing~$p$;
 \item places above an auxiliary rational prime $q \in \TT$ that is split in $E/\Q$ and has order ${\elll}$ modulo $p$; (A positive density of such $q$ exist since $E$ is linearly disjoint from $\Q(\mu_p)$ over $\Q$, by comparing ramification at $p$.)
 \item an auxiliary place $w$ lying above a rational prime $r$ that splits completely in $E(\zeta_p)/\Q$.
\end{itemize}
Consider the following local classes in $H^1(\Gammma_{F^+_v}, T({\F}_p)[{\elll}])$ for $v \in T$:
\begin{itemize}
 \item trivial at places which ramify in $F/F^+$;
 \item the prescribed homomorphism $\prod_{\alpha \in \Delta} \alpha^\vee \circ \bar{\kapppa}[{\elll}]^{n_{v, \alpha}}$ for $v \vert p$,
 \item the unramified homomorphism $\mathrm{Fr}_v \mapsto \rho_G^\vee(q)$ for $v \vert q$ and~$q \in \TT$, where $\rho_G^\vee$ denotes the half-sum of the positive coroots of $G$, which is in fact a cocharacter of $G$ (we will only need this construction for one of the places above $q$).	
 \item the unramified homomorphism $\mathrm{Fr}_w \mapsto t$, where $t$ is any element of $T({\F}_p)[{\elll}]$ such that the values $\lambda(t)$ for $\lambda \in \Lambda_{\mr{min}}$ are all distinct (for ${\elll}$ sufficiently large, such $t$ exist). 
  \end{itemize}
Let $\phi \in H^1(\Gammma_{F^+, \Sigma}, T({\F}_p)[{\elll}])$ be a class with these local restrictions, and set $\rhobar= \phi \cdot s$. We claim the conclusions of the proposition hold for this $\rhobar$. The conditions at finite places are all evident from the construction, and the condition on complex conjugation is satisfied because any order two element 
$(\ww_0, \tau) \in N_G(T)({\F}_p)$ 
lifting $(w_0, \tau) \in W_G \rtimes \Out(G)$ gives a split Cartan involution of $\mathfrak{g}$: it acts by $-1$ on $\mathrm{Lie}(T)$, and it sends a root space $\mathfrak{g}_{\alpha}$ to $\mathfrak{g}_{-\alpha}$, so we get precisely $\dim(G/B)= \dim(\mathfrak{g}^{\mathrm{Ad}((\ww_0, \tau))=1})$.

For the second list of assertions, note that for $v \vert p$, $r_{\mathrm{min}}\circ \rhobar|_{\Gammma_{F^+_v}}$ is equal to
\[
\bigoplus_{\lambda \in \Lambda_{\mathrm{min}}} \bar{\kapppa}[{\elll}]^{\sum_{\alpha \in \Delta} n_{v, \alpha}\langle \lambda, \alpha^\vee \rangle},
\]
where $\Lambda_{\mathrm{min}}$ is the set of weights of $r_{\mathrm{min}}$. For varying $\lambda$, we want these exponents to be distinct modulo ${\elll}$.
We simply choose the $(n_{v, \alpha})_{\alpha}$ so that the exponents are distinct in $\Z$, and then any ${\elll}$ sufficiently large will do. To evaluate the cohomology groups appearing in the conclusion of the proposition, note that as a $\Gammma_{F^+_v}$-module, $\rhobar_v(\mathfrak{g}/\mathfrak{b})$ is a direct sum of characters of the form
\[
 \bar{\kapppa}[{\elll}]^{\sum_{\alpha \in \Delta} n_{v, \alpha} \langle \beta, \alpha^\vee \rangle},
\]
where $\beta$ is a negative root of $G$. Clearly the choice of ${\elll}$ can be modified if necessary to ensure that these exponents (which don't depend on $p$) are all integers between $1$ and ${\elll} -1$, and so regardless of how $p$ is chosen the group $H^0(\Gammma_{F^+_v}, \rhobar_v(\mathfrak{g}/\mathfrak{b}))$ will vanish. The vanishing of $H^0(\Gammma_{F^+_v}, \rhobar_v(\mathfrak{g}/\mathfrak{b})(1))$ is even more straightforward: the order of $\bar{\kapppa}$ is divisible by some prime other than ${\elll}$, whereas all powers of $\bar{\kapppa}[{\elll}]$ have order $1$ or ${\elll}$. We can similarly deduce the vanishing of $H^0(\Gammma_{F^+_v}, (r_{\mathrm{min}}\circ \rhobar_v)(\mathfrak{sl}_{27}/\mathfrak{b}_{27}))$ and $H^0(\Gammma_{F^+_v}, (r_{\mathrm{min}}\circ \rhobar_v)(\mathfrak{sl}_{27}/\mathfrak{b}_{27})(1))$, where $\mathfrak{b}_{27}$ is the Lie algebra of any Borel subgroup $B_{27} \subset \mathrm{SL}_{27}$ containing the maximal torus of $\mathrm{SL}_{27}$ characterized by the property that it stabilizes each of the weight spaces of $r_{\mathrm{min}}$. Then $(r_{\mathrm{min}}\circ\rhobar_v)(\mathfrak{sl}_{27}/\mathfrak{b}_{27})$ is a direct sum of characters
\[
 \bar{\kapppa}[{\elll}]^{\sum_{\alpha \in \Delta} n_{v, \alpha}\langle \lambda_1-\lambda_2, \alpha^\vee \rangle},
\]
for distinct weights $\lambda_1, \lambda_2 \in \Lambda_{\mathrm{min}}$. These exponents are by construction non-zero modulo ${\elll}$, so we win, and the same argument as for $\rhobar_v(\mathfrak{g}/\mathfrak{b})(1)$ applies to show $H^0(\Gammma_{F^+_v}, (r_{\mathrm{min}}\circ\rhobar_v)(\mathfrak{sl}_{27}/\mathfrak{b}_{27})(1))$ is zero as well. 

The condition at the auxiliary prime $w$ ensures absolute irreducibility of $r_{\mathrm{min}} \circ \rhobar|_{\Gammma_{F(\zeta_p)}}$: $P$ acts transitively on $\Lambda_{\mathrm{min}}$ by
Lemma~\ref{lemma:split}(\ref{splittransitive}), so any non-zero submodule of $r_{\mathrm{min}} \circ \rhobar|_{\Gammma_{F(\zeta_p)}}$ has non-zero projection to each weight space (recall that $F(\zeta_p)$ is linearly disjoint from $E$ over $F$); but for a place $w' \vert w$ of $F(\zeta_p)$, the image of $\Gammma_{F(\zeta_p)_{w'}}$ in $T({\F}_p)$ acts via distinct characters on the different weight spaces of $r_{\mathrm{min}}$.
\end{proof}

\section{Lifting Galois Representations}

Let $\rhobar \colon \Gammma_{F^+} \to N_G(T)({\F}_p) \rtimes \Out(G) \subset {}^L G({\F}_p)$ be a homomorphism constructed as in Proposition \ref{seed}. We would like to lift $\rhobar$ to a homomorphism 
\[
 \rho \colon \Gammma_{F^+} \to {}^L G(\Zbar_p)
\]
that belongs to a compatible system of representations, all having Zariski-dense image and appearing in the cohomology of an algebraic variety. To achieve Zariski-dense monodromy for a lift $\rho$, we follow the approach of \cite{Patrikis}: ensuring local Steinberg-type ramification at one auxiliary prime and sufficiently general Hodge--Tate cocharacter suffices. To produce the lift, and to put it in a compatible system, we compare deformation rings for $\rhobar$ and $r_{\mathrm{min}}\circ \rhobar \colon \Gammma_{F^+} \to \mathcal{G}_{27}({\F}_p)$. We control a suitable deformation ring for $r_{\mathrm{min}}\circ \rhobar$ using the method of Khare--Wintenberger and the automorphy lifting results of \cite{BLGGT}.

Let $S$ be a finite set of primes of $F^+$ which split in $F$ containing all those where $\rhobar$ is ramified (by the construction of $\rhobar$, the primes of $F^+$ at which $\rhobar$ is ramified split in $F$). We will enlarge the set $S$ as necessary. Let $F_S$ be the maximal extension (in $\overline{F}^+$) of $F$ unramified outside (places above) $S$, and set $\Gammma_S= \Gal(F_S/F^+)$. We will be deforming $\Gammma_S$-representations. For each $v \in S$, fix an extension~$\tilde{v}$ of~$v$ to~$F$, and fix a member of the $\Gammma_{F, S}$-conjugacy class of homomorphisms $\Gammma_{F_{\tilde{v}}} \to \Gammma_{F, S}$. Via the inclusion $\Gammma_{F, S} \subset \Gammma_S$, these choices specify what we mean by restricting $\rhobar$ (or its lifts) to $\Gammma_{F_{\tilde{v}}}$. We do not review in detail the mechanics of the deformation theory in this setting, but instead refer the reader to \cite[\S 9.2]{Patrikis} (for ${}^L G$) and \cite[\S 2]{CHT} (for $\mathcal{G}_{27}$). We now specify local deformation conditions to define two global deformation functors, one for $\rhobar$ and one for $r_{\mathrm{min}}\circ \rhobar$.

For $v \in S$, define the following local deformation conditions $\mathcal{P}_v$ on lifts of $\rhobar|_{\Gammma_{F_{\tilde{v}}}}$:
\begin{itemize}
 \item For $v$ above the auxiliary primes $q \in \TT$ (see Proposition \ref{seed}), we impose the Steinberg deformation condition as in \cite[\S 4.3]{Patrikis}, with respect to the Borel subgroup $B$ of $G$ specified by our based root datum. Note that by construction the order of $\bar{\kapppa} \colon \Gammma_{F_{\tilde{v}}} \to {\F}_p^\times$ is greater than $h_G-1$ (since ${\elll}>h_G$ is the order of $q$ modulo $p$).
 \item For $v \vert p$, we take an ordinary deformation condition as in \cite[\S 4.1]{Patrikis}. To be precise, fix the following lift $\chi_T$ of $\rhobar|_{I_{F_{\tilde{v}}}}$ to $T(\Z_p)$:
\begin{equation}\label{HTweights}
 \chi_T= \prod_{\alpha \in \Delta} \alpha^\vee \circ(\kapppa^{\tilde{n}_{v, \alpha}}\cdot \chi^{-n_{v, \alpha}}),
\end{equation}
where the $\tilde{n}_{v, \alpha}$ are sufficiently general positive (positive ensures, as in \cite[Lemma 4.8]{Patrikis}, that our characteristic zero lifts are de Rham) integers congruent to $n_{v, \alpha}$ modulo ${\elll}-1$, and $\chi$ is the Teichm\"{u}ller lift of $\bar{\kapppa} \cdot  \bar{\kapppa}[{\elll}]^{-1}$.
 \item For all other primes $v \in S$, the inertial image $\rhobar(I_{F_{\tilde{v}}})$ has order prime to $p$ (indeed, $\rhobar$ lands in the prime-to-$p$ group $N_G(T)({\F}_p) \rtimes \Out(G)$), and we take the minimal deformation condition of \cite[\S 4.4]{Patrikis}. 
\end{itemize}
\begin{lemma}\label{Gdef}
For all $v \in S$, let $\mathcal{P}_v$ be the local condition just defined. Let $\mathcal{P}= \{\mathcal{P}_v\}_{v \in S}$, and let $\mathrm{Lift}_{\rhobar}^{\mathcal{P}}$ be the associated global lifting functor associated to this collection of local conditions (see \cite[\S 9.2]{Patrikis}).
\begin{enumerate}
 \item
 For all $v \in S$ not above $p$, the local lifting ring associated to the condition $\mathcal{P}_v$ just defined is formally smooth, and the associated deformation functor has tangent space $L_v$ of dimension $\dim H^0(\Gammma_{F_{\tilde{v}}}, \rhobar(\mathfrak{g}))$. For $v \vert p$, the same holds, except the tangent space has dimension $\dim H^0(\Gammma_{F_{\tilde{v}}}, \rhobar(\mathfrak{g}))+ \dim (G/B)$.
 \item The associated deformation functor $\mathrm{Def}^{\mathcal{P}}_{\rhobar}$ is representable. Let $R^{\mathcal{P}}_{\rhobar}$ be the representing object. For some integer~$\delta$, $R^{\mathcal{P}}_{\rhobar}$ has a presentation as the quotient of a power series ring over~$\Z_p$ in~$\delta$ variables by an ideal generated by (at most)~$\delta$ relations.
\end{enumerate}
\end{lemma}
\begin{proof}
The local claims follow from \cite[\S 4.1, 4.3, 4.4]{Patrikis}. The global claims follow from \cite[Proposition 9.2]{Patrikis}), using that:
\begin{itemize}
 \item the centralizer of $\rhobar$ in $\mathfrak{g}$ is trivial;
\item for all complex conjugations $c$, $\rhobar(c)$ is a split Cartan involution of $G$;
\item the local lifting rings have dimensions as computed in the first part of the lemma.
\end{itemize}
(We remark that the integer~$\delta$ is the common dimension of the Selmer and dual Selmer groups associated to the global deformation functor.) 
\end{proof}
Next we define an analogous deformation ring for $r_{\mathrm{min}}\circ \rhobar \colon \Gammma_S \to \mathcal{G}_{27}({\F}_p)$. Recall the character $\nu \colon \mathcal{G}_{27} \to \mathbf{G}_m$. The composition $\nu \circ (r_{\mathrm{min}}\circ \rhobar)$ is the non-trivial character of $\Gal(F/F^+)$, and we fix $\mu \colon \Gammma_{F^+} \to \Z_p^{\times}$ equal to its Teichm\"{u}ller lift (we will consider lifts with this fixed character). To define a global deformation problem in the sense of \cite[\S 1.5]{BLGGT} (see \cite{CHT} for more details), we must, for each $v \in S$, 
choose an irreducible component $\mathcal{C}_v$ of the (generic fiber) lifting ring (in the case $v$ not above $p$) $R^{\square}_{r_{\mathrm{min}} \circ \rhobar|_{\Gammma_{F_{\tilde{v}}}}}[1/p]$ or (in the case $v \vert p$) $\lim_K 
R^{\square}_{r_{\mathrm{min}} \circ \rhobar|_{\Gammma_{F_{\tilde{v}}}}, \{H\}, K-\mathrm{ss}}[1/p]$, where here we follow the notation of \cite{BLGGT}: $H$ is a collection (indexed by embeddings $F_{\tilde{v}} \to \Q_p$, but for us $F_{\tilde{v}}= \Q_p$) of multi-sets of Hodge numbers, $K$ varies over finite extensions of $F_{\tilde{v}}$, and the lifting ring in question is the one constructed by Kisin (\cite{Kisin}), whose characteristic zero points parametrize potentially semistable deformations, semistable over $K$, with the prescribed Hodge numbers; to be precise, it is the maximal reduced $p$-torsion-free quotient of $R^{\square}_{r_{\mathrm{min}} \circ \rhobar|_{\Gammma_{F_{\tilde{v}}}}}$ whose $\Qbar_p$-points satisfy these properties (see \cite[\S 1.4]{BLGGT} for an overview). In both cases $v \vert p$ and $v \nmid p$, we then associate the lifting ring $R^{\mathcal{C}_v}_{r_{\mathrm{min}} \circ \rhobar|_{\Gammma_{F_{\tilde{v}}}}}$ given by the maximal reduced $p$-torsion-free quotient of ($v \nmid p$) $R^{\square}_{r_{\mathrm{min}} \circ \rhobar|_{\Gammma_{F_{\tilde{v}}}}}$ or ($v \vert p$) $\lim_K 
R^{\square}_{r_{\mathrm{min}} \circ \rhobar|_{\Gammma_{F_{\tilde{v}}}}, \{H\}, K-\mathrm{ss}}$ that is, after inverting $p$, supported on the component $\mathcal{C}_v$. Namely, we take:
\begin{itemize}
 \item For $v$ above the auxiliary primes $q \in \TT$, recall that $\rhobar|_{\Gammma_{F_{\tilde{v}}}}$ is unramified with (arithmetic) Frobenius $\mathrm{Fr}_v$ mapping to $\rho_G^{\vee}(q)$. Let $\varphi \colon \mathrm{PGL}_2 \to G$ be the principal homomorphism associated to our fixed pinning of $G$, so $\rhobar(\mathrm{Fr}_v)$ equals $\varphi(\mathrm{diag}(q, 1))$. The composite $r_{\mathrm{min}} \circ \varphi$ decomposes as $\mathrm{S}^{16} \oplus \mathrm{S}^{8} \oplus \mathrm{S}^0$ (see \cite[\S 7]{gross}), where we write $S^i$ for the $i^{th}$ symmetric power of the standard representation of $\mathrm{SL}_2$. Consider the lift $\rho_{\tilde{v}}$ of $r_{\mathrm{min}} \circ \rhobar|_{\Gammma_{F_{\tilde{v}}}}$ given by $r_{16} \oplus r_8 \oplus r_0$ where $r_i$ is the tame unipotent representation on the $\mathrm{S}^i$ component given by the matrices
\[
\left(r_i(\mathrm{Fr}_v)\right)_{a, b}= 
\begin{cases}
 q^{i-2a+2} & \text{if $a=b$,} \\
0 &\text{if $a \neq b$,} 
\end{cases}
\]
and on a topological generator $\tau_v$ of tame inertia by
\[
 \left(r_i(\tau_v)\right)_{a, b}=
\begin{cases}
 p & \text{if $b=a+1$,} \\
0 & \text{if $b \neq a+1$.}
\end{cases}
\]
A quick calculation shows that, for all finite extensions $K/F_{\tilde{v}}$, we have an
 equality~$H^0(\Gammma_{K}, \mathrm{ad}(\rho_{\tilde{v}})(1))=0$, so $\rho_{\tilde{v}}$ is a robustly smooth point of $R^{\square}_{r_{\mathrm{min}} \circ \rhobar|_{G_{F_{\tilde{v}}}}}[1/p]$, in the sense of \cite[\S 1.3]{BLGGT}. In particular, $\rho_{\tilde{v}}$ lies on a unique irreducible component of $R^{\square}_{r_{\mathrm{min}} \circ \rhobar|_{\Gammma_{F_{\tilde{v}}}}}[1/p]$, and we take $\mathcal{C}_v$ to be this component. 
 \item For $v \vert p$, let the set $H$ of Hodge--Tate weights be 
\[
\{h_{\lambda}= \sum_{\alpha \in \Delta} \tilde{n}_{v, \alpha} \langle \lambda, \alpha^\vee \rangle \}_{\lambda \in \Lambda_{\mathrm{min}}},
\]
where $\Lambda_{\mathrm{min}}$ is the set of weights of $r_{\min}$ as before. Borel subgroups of $\mathrm{SL}_{27}$ containing the maximal torus $T_{27}$ (the unique torus stabilizing the weight spaces in $r_{\mathrm{min}}$) are in bijection with orderings of the set $\Lambda_{\mathrm{min}}$; let $B_{27}$ be the Borel defined by the ordering $\lambda > \lambda' \iff h_{\lambda} > h_{\lambda'}$. The ordinary deformation ring (again following the notation of \cite[\S 4.1]{Patrikis}) associated to the Borel $B_{27}$ and the lift $\chi_{T_{27}} \colon I_{F_{\tilde{v}}} \to T_{27}(\Z_p)$ of $r_{\mathrm{min}}\circ \rhobar|_{\Gammma_{F_{\tilde{v}}}}$ given by
\[
 \chi_{T_{27}}= r_{\mathrm{min}}\left(\prod_{\alpha \in \Delta} \alpha^\vee(\kapppa^{\tilde{n}_{v, \alpha}} \chi^{-n_{v, \alpha}})\right)
\]
is formally smooth and receives by the universal property a canonical surjection from $R^{\square}_{r_{\mathrm{min}}\circ \rhobar|_{\Gammma_{F_{\tilde{v}}}}, H, K-\mathrm{ss}}$ for $K= F_{\tilde{v}}(\chi)$ (recall that $\chi$ is the Teichm\"{u}ller lift of $\bar{\kapppa} \cdot (\bar{\kapppa}[{\elll}])^{-1}$), and after inverting $p$ it induces an isomorphism from a unique irreducible component of the source (the equality of dimensions follows from \cite{Kisin} and \cite[\S 4.1]{Patrikis}). We take $\mathcal{C}_v$ to be this component.
 \item For all other $v \in S$, we take the irreducible component of $R^{\square}_{r_{\mathrm{min}}\circ \rhobar|_{\Gammma_{F_{\tilde{v}}}}}$ parametrizing minimal deformations in the sense of \cite[\S 4.4]{Patrikis}; using the standard argument, this is a power series ring and induces a unique irreducible component $\mathcal{C}_v$ of $R^{\square}_{r_{\mathrm{min}}\circ \rhobar|_{\Gammma_{F_{\tilde{v}}}}}[1/p]$.
\end{itemize}
In each case, we write $R_{r_{\mathrm{min}}\circ \rhobar|_{\Gammma_{F_{\tilde{v}}}}}^{\mathcal{C}_v}$ for the associated local lifting ring. Recall (\cite[\S 1.3]{BLGGT}) that for $v$ not above $p$, $R_{r_{\mathrm{min}}\circ \rhobar|_{\Gammma_{F_{\tilde{v}}}}}^{\mathcal{C}_v}$ is the maximal quotient of $R^{\square}_{r_{\mathrm{min}}\circ \rhobar|_{\Gammma_{F_{\tilde{v}}}}}$ that is reduced, $p$-torsion-free, and after inverting $p$ is supported on the component $\mathcal{C}_v$. For $v \vert p$, $R^{\mathcal{C}_v}_{r_{\mathrm{min}}\circ \rhobar|_{\Gammma_{F_{\tilde{v}}}}}$ is in general constructed similarly, but for us it is simply the formally smooth ordinary deformation ring produced by \cite[\S 4.1]{Patrikis}.
\begin{lemma}\label{localliftmap}
 For all $v \in S$, the representation $r_{\mathrm{min}}$ induces a map $R_{r_{\mathrm{min}}\circ \rhobar|_{\Gammma_{F_{\tilde{v}}}}}^{\mathcal{C}_v} \to R_{\rhobar|_{\Gammma_{F_{\tilde{v}}}}}^{\mathcal{P}_v}$
\end{lemma}
\begin{proof}
 For $v \vert p$, this follows directly from the definitions once we check that $r_{\mathrm{min}}(B) \subset B_{27}$. Let $\lambda= \sum_{\alpha \in \Delta} \tilde{n}_{v, \alpha} \alpha$. Since all $\tilde{n}_{v, \alpha}$ are positive, $B$ is the locus of $g \in G$ where $\lim_{t \to 0} \mathrm{Ad}(\lambda(t))g$ exists. By construction, $B_{27}$ is the locus where $\lim_{t \to 0} \mathrm{Ad}(r_{\mathrm{min}}\circ \lambda(t))g$ exists. The claim follows.

For $v \vert q$, we write $R^{\mathcal{P}_v}_{\rhobar|_{\Gammma_{F_{\tilde{v}}}}}$ for the Steinberg lifting ring; recall that under our hypotheses it is a power series ring over $\Z_p$ (in $\dim(\mathfrak{g})$ variables). We have surjections
\[
 R^{\square}_{r_{\mr{min}}\circ \rhobar|_{\Gammma_{F_{\tv}}}} \to R^{\square}_{\rhobar|_{\Gammma_{F_{\tv}}}} \to R^{\mathcal{P}_v}_{\rhobar|_{\Gammma_{F_{\tv}}}}.
\]
Recall that $R^{\mathcal{C}_v}_{r_{\mr{min}}\circ \rhobar|_{\Gammma_{F_{\tv}}}}$ is the maximal reduced $p$-torsion-free quotient of $R^{\square}_{r_{\mr{min}}\circ \rhobar|_{\Gammma_{F_{\tv}}}}$ whose $\Qbar_p$-points lie on $\mathcal{C}_v$. Since $R^{\mathcal{P}_v}_{\rhobar|_{\Gammma_{F_{\tv}}}}$ is a reduced, $p$-torsion-free quotient of $R^{\square}_{r_{\mr{min}}\circ \rhobar|_{\Gammma_{F_{\tv}}}}$, it suffices to show that every $\Qbar_p$-point of $R^{\mathcal{P}_v}_{\rhobar|_{\Gammma_{F_{\tv}}}}$ lies on the same irreducible component of $R^{\square}_{r_{\mr{min}}\circ \rhobar|_{\Gammma_{F_{\tv}}}}[1/p]$ as the representation $\rho_{\tv}$ (defined above) that characterizes the component $\mathcal{C}_v$. This claim follows from \cite[Lemma 1.3.5]{BLGGT}. 

For the other ramified primes $v \in S$, the lemma is evident.  
\end{proof}
Now consider the global deformation problem (in the sense of \cite[\S 1.5]{BLGGT}; see \cite{CHT} for details)
\[
 \mathcal{S}=(F/F^+, S, \{\tilde{v}\}_{v \in S}, \Z_p, r_{\mathrm{min}}\circ\rhobar, \mu, \{\mathcal{C}_v\}_{v \in S}).
\]
Since $r_{\mathrm{min}}\circ\rhobar|_{\Gammma_F}$ is absolutely irreducible, this deformation functor is pro-represented by some $R^{\mathcal{S}}_{r_{\mathrm{min}}\circ\rhobar}$.
\begin{lemma}\label{finite}
 The representation $r_{\mathrm{min}}$ induces a surjection $R^{\mathcal{S}}_{r_{\mathrm{min}}\circ\rhobar} \to R_{\rhobar}^{\mathcal{P}}$.
\end{lemma}
\begin{proof}
 There is an induced map $R^{\mathcal{S}}_{r_{\mathrm{min}}\circ\rhobar} \to R_{\rhobar}^{\mathcal{P}}$ by Lemma \ref{localliftmap}. It is a surjection because the $\Gammma_{F^+}$-module $\rhobar(\mathfrak{g})$ is a direct summand of $\mathrm{ad}(r_{\mathrm{min}}\circ \rhobar)$ (indeed, these representations factor through representations of a finite prime-to-$p$ group), so the associated map on tangent spaces is injective; dually, the map on co-tangent spaces is surjective, and we conclude by Nakayama's lemma.
\end{proof}
Finally, we can invoke the main results of \cite{BLGGT} to deduce that $R^{\mathcal{P}}_{\rhobar}$ has a $\Qbar_p$-point $\rho$ such that $r_{\mr{min}}\circ \rho$ is potentially automorphic:
\begin{theorem} \label{theorem:lifts}
 For sufficiently general choice of lifts $\tilde{n}_{v, \alpha}$ as in Equation \ref{HTweights}, the representation $\rhobar \colon \Gammma_{F^+} \to {}^L G({\F}_p)$ constructed in Proposition \ref{seed} admits a geometric lift $\rho \colon \Gammma_{F^+} \to {}^L G(\Zbar_p)$ such that:
\begin{enumerate}
 \item The Zariski closure of the image of $\rho$ is ${}^L G$. \label{lifts:zariski}
 \item The composite $r_{\mr{min}}\circ \rho|_{\Gammma_F}$ is potentially automorphic in the sense of \cite{BLGGT}. \label{part:automorphic}
   \item The composite $r_{\mr{min}}\circ \rho|_{\Gammma_F}$ belongs to a compatible system of ${\elll}$-adic representations: there exist a number field $M$ and a strictly pure  --- in the sense of \cite[\S 5.1]{BLGGT}
    --- compatible system $$r_{\lambda} \colon \Gammma_F \to \mr{GL}_{27}(\overline{M}_{\lambda})$$
   indexed over all finite places $\lambda$ of $M$. 
   In particular,
   the restriction of~$r_{\lambda}$ to~$I_v$ for~$v$ above the auxiliary primes~$q \in \TT$ is unipotent with Jordan blocks of size~$1$, $9$, and~$17$ as long as~$v$
   has residue characteristic different from~$\lambda$. \label{grossblock}
   \end{enumerate}

\end{theorem}
 \begin{proof}
  By the proof of \cite[Theorem 4.3.1]{BLGGT} (see especially the last paragraph), $R^{\mathcal{S}}_{r_{\mr{min}}\circ \rhobar}$ is a finite $\Z_p$-module. Lemma \ref{finite} then implies that $R^{\mathcal{P}}_{\rhobar}$ is a finite $\Z_p$-module. We have already seen in Lemma \ref{Gdef} that it has dimension at least one, so we conclude that $R^{\mathcal{P}}_{\rhobar}(\Zbar_p)$ is non-empty. Let $\rho$ be an element of $R^{\mathcal{P}}_{\rhobar}(\Zbar_p)$. Then:
\begin{itemize}
 \item By \cite[Theorem 4.5.1]{BLGGT}, the composite $r_{\mr{min}}\circ \rho|_{\Gammma_{F}}$ is potentially automorphic (in the sense of \cite{BLGGT}).
 \item The Zariski closure $G_{\rho}$ of the image of $\rho|_{\Gammma_F}$ is $G$: by \cite[Lemma 7.8]{Patrikis}, it suffices to show 
\begin{itemize} \item $G_{\rho}$ is reductive; 
\item $G_{\rho}$ contains a regular unipotent element of $G$; and
\item for some $v \vert p$, $\rho|_{\Gammma_{F_{\tv}}}$ is $B$-ordinary, and, for all simple roots $\alpha$, $\alpha \circ \rho|_{I_K}= \kapppa^{r_{\alpha}}$ for some finite extension $K/F_{\tv}$ and for distinct integers $r_{\alpha}$.
\end{itemize}
Reductivity is immediate since $\rho$ is irreducible. The third condition follows by taking the integers $\{\tilde{n}_{v, \alpha}\}_{\alpha \in \Delta}$ in the definition of the local condition $\mathcal{P}_v$ (see the discussion preceeding Lemma \ref{Gdef}) to be sufficiently general. Finally, to show that the image of $\rho$ contains a regular unipotent element, we check that for $v \vert q$, the tame inertia in $\rho|_{\Gammma_{F_{\tv}}}$ acts by a regular unipotent. This would follow from the corresponding claim that $r_{\mr{min}}\circ \rho(I_{F_{\tv}})$ contains a unipotent element with Jordan blocks of dimension 17, 9, and 1. Let $\pi$ be the automorphic representation of $\mr{GL}_{27}(\mathbf{A}_{F'})$, for a suitable finite extension $F'/F$, witnessing the potential automorphy of $r_{\mr{min}}\circ \rho$, and let $v'$ be a place of $F'$ above $v$. By
local-global compatibility at~${\elll} \neq p$ (Proved in general by~\cite{Caraiani},
but known for odd dimensional representations by
previous work of~\cite{harris-taylor,Yoshida,MR2800722}), the (Frobenius semi-simple) Weil--Deligne representation associated to $r_{\mr{min}}\circ\rho|_{\Gammma_{F'_{v'}}}$ is isomorphic to the image of $\pi_{v'}$ under the local Langlands correspondence. It follows (eg, using \cite[Lemma 1.3.2(1)]{BLGGT}) that $r_{\mr{min}}\circ \rho|_{\Gammma_{F_{\tv}}}$ lies on a unique irreducible component of $R^{\square}_{r_{\mr{min}}\circ \rho|_{\Gammma_{F_{\tv}}}}$. By construction, $r_{\mr{min}}\circ \rho|_{\Gammma_{F_{\tv}}}$ and $r_{\mr{min}} \circ \rho_{\tv}$ lie on the same irreducible component $\mathcal{C}_v$ of $R^{\square}_{r_{\mr{min}}\circ \rhobar|_{\Gammma_{F_{\tv}}}}$, and since they both lie on a unique component, \cite[Lemma 1.3.4(2)]{BLGGT} implies their inertial restrictions are isomorphic. The result follows.

The claim that $r_{\mr{min}}\circ \rho$ can be put in a compatible system, follows from \cite[Theorem 5.5.1]{BLGGT}, and the claim concerning the restriction to~$I_v$ for~$v|q$
follows as above from local-global compatibility at~$\elll \ne p$.
\end{itemize}
 \end{proof}
In the next section, we will show that in fact all members of the compatible system $\{r_{\lambda}\}_{\lambda}$ have algebraic monodromy group equal to $G$.

\section{Controlling the image in the compatible system}

Theorem~\ref{theorem:lifts} provides the existence of a compatible system~$\{r_{\lambda}\}$  of~$G_{F}$ representations with
the property that the geometric monodromy group at one
prime is precisely~$E_6$.
Our goal in this section is to use known properties of compatible systems (\cite{LP}) together with the additional
properties our compatible system satisfies at the auxiliary primes~$\TT$ to ensure that the monodomy group
is~$E_6$ at~\emph{all} primes~$\lambda$.

  Let~$M$ denote the coefficient field of our compatible system.
  
  \begin{lemma} \label{lemma:stuff} The monodromy group~$G$  for each prime~$\lambda$ of~$M$ has the following properties:
 \begin{enumerate} 
 \item The component group of~$G$ is is trivial.
 \item The rank of~$G$ is~$6$.
 \item The formal character of the torus~$\chi: T \rightarrow \GL_{27}$ is the
 formal character of the torus of~$E_6$ under the minuscule representation.
 \item If~$G = G^{\circ}$ acts irreducibly, then~$G$ is equal to~$E_6$.
 \item There exists a unipotent element in the image with Jordan blocks of size~$1$, $9$, and~$17$. \label{part:stuff}
 \end{enumerate}
 \end{lemma}
 
 \begin{proof}
The first three properties involve quantities which 
are constant in a compatible system, c.f.~Propositions~6.12 and~6.14 of~\cite{LP}.
The fourth claim follows  from Theorem~5.6 of~\cite{LP}, noting that~$E_6$ does not occur
in the explicit list of groups which gives rise to any of the basic similarity relations of~\S5 of \emph{ibid}.
This is enough to deduce that the Lie algebra of the monodromy representation must be~$\mathfrak{e}_6$,
from which it follows that~$G$ is~$E_6$ (acting in the natural way).
The nilpotent operator of the Weil--Deligne representation of~$r_{\lambda}$ at the auxiliary prime~$v|q$ for~$q \in \TT$
 decomposes (by Theorem~\ref{theorem:lifts} part~\ref{grossblock})
 into Jordan blocks of size~$1$, $9$, and~$17$, \emph{assuming} that the residual characteristic of~$q$ is different from that of~$\lambda$.
  Yet~$\TT$ was chosen (for this purpose!) to consist of two primes, so this holds for at least one prime~$v$.
\end{proof}

We now show that these conditions are sufficient to imply --- purely by representation theoretic methods --- that~$G$ acts irreducibly, which will prove the claims in Theorem~\ref{theorem:main} concerning the monodromy groups of~$\{r_{\lambda}\}$.
In light of the existence of  the unipotent element whose existence is guaranteed by Lemma~\ref{lemma:stuff} part~\ref{part:stuff}, it suffices to show that~$G$ cannot act faithfully on a direct
sum of representations of dimension
$$27 = 26 + 1 = 18 + 9 = 17 + 10 = 17+ 9 + 1$$
unless~$G = E_6$ and the representation is irreducible.

 \subsection{The Formal Character of~\texorpdfstring{$\e_6$}{E6}}
 The root lattice~$\Phi$ of~$E_6$ consists of~$72$ roots; it may be given as~$\Phi^{+} \cup \Phi^{-},$
 where the positive roots~$\Phi^{+}$ are given explicitly in~$\R^6$ by
 the~$\displaystyle{2 \binom{5}{2}}= 20$ vectors~$e_{i} \pm e_{j}$ for~$2 \le i < j \le 6$, and
 the~$2^4 = \displaystyle{ \binom{5}{0} + \binom{5}{2} + \binom{5}{4}} = 16$ vectors
$$\left(\frac{\sqrt{3}}{2}, \frac{\pm 1}{2}, \frac{\pm 1}{2},
\frac{\pm 1}{2}, \frac{\pm 1}{2}, \frac{\pm 1}{2}\right)$$
where there are an even number of minus signs. If~$2 \rho = \sum_{\alpha \in \Phi^{+}} \alpha$,
then
$\rho = (4 \sqrt{3},4,3,2,1,0)$.  The root lattice is not self-dual, but has discriminant~$3$.  
A weight~$\mu$ corresponding to a choice of minuscule representation is given by
$$\mu = \frac{1}{3} (2 \sqrt{3},0,0,0,0,0).$$
 The~$27$ weights~$\Sigma$ of
the corresponding minuscule representation may be obtained from~$\mu$ from
the orbit of the Weyl group; all~$27$ such weights may be obtained by applying
at most~$2$ reflections in the roots of~$\Phi$ to~$\mu$.
We have the following: (cf.~\cite{Lurie})
Of the~$\binom{27}{3} = 2925$ collections of~$3$ vectors in~$\Sigma$, exactly~$45$ such triples
generate  a subspace of dimension~$2$, and they all consist of a triple
of weights~$(\mu,\mu',\mu'')$ with~$\mu + \mu' + \mu'' = 0$.
If~$\Lambda$ is the weight lattice, then~$\Sigma$ injects into~$V = \Lambda/2 \Lambda$,
which acquires the structure of a quadratic space via the map~$q(\mu) = \frac{1}{2}
\langle \mu,\mu \rangle$ (note that~$\Lambda$ is an even lattice). The 
pairing~$\langle x,y \rangle = q(x+y) - q(x) - q(y)$
is preserved by the Weyl group~$W_{G}$, which may be identified with the corresponding
orthogonal group. The lattice~$V$
also admits a Hermitian structure corresponding to~$q$. With respect
to this structure, the quadratic space~$V$
has Arf invariant~$1$, and the elements~$\{\mu,\mu',\mu''\}$ above
lie inside a maximal isotropic subspace~$U \subset V$ of dimension~$2$.
The stabilizer of~$U$ (and of a triple) is a subgroup of~$W_G$ of index~$45$,
which correspondingly acts transitively on the set of~$45$ triples.
The stabilizer is also isomorphic to the Weyl group of~$F_4$.
An explicit example of a triple is given by
$$
\left(\begin{matrix} \mu  \\  \mu' \\ \mu''  \end{matrix}\right)
 =  \frac{1}{3}
\left( \begin{matrix} 2 \sqrt{3} & 0 & 0 & 0 & 0 & 0 \\
- \sqrt{3} & +3  &  0 & 0 & 0 & 0 \\
- \sqrt{3} & -3  &   0 & 0 & 0 & 0 \\ \end{matrix} \right).$$
where
$\mu' = \sigma_{\alpha} \sigma_{\beta} \mu$, $\mu'' = \sigma_{\gamma} \sigma_{\delta} \mu$, and
$$\begin{aligned}
\alpha = & \ (\sqrt{3}/2, -1/2, -1/2, 1/2, 1/2, 1/2) \\
\beta = & \ (\sqrt{3}/2, -1/2, 1/2, -1/2, -1/2, -1/2) \\
\gamma = & \ (\sqrt{3}/2, 1/2, -1/2, -1/2, 1/2, 1/2) \\
\delta = & \ (\sqrt{3}/2, 1/2, -1/2, 1/2, -1/2, -1/2). \\
\end{aligned}
$$

We derive the following  consequence:

\begin{lemma} \label{corr:quote} The restriction of~$G$ to any sub-representation 
of dimension~$\ge  4$ must factor through a quotient of rank at least~$3$, and the restriction
of~$G$ to any sub-representation of dimension~$26$ must have rank~$6$.
\end{lemma}

\begin{proof} Given four distinct weights of a sub-representation on which the action of $G$ factors through a quotient of rank at most $2$, any three of them must consist of a triple~$(\mu,\mu',\mu'')$ which sum to zero, which cannot hold for more than one such triple. 
One can also prove this by a direct explicit computation. The second claim follows
 from the fact that the sum of all~$27$ weights in~$\Sigma$ is zero,
and none of the weights in~$\Sigma$ is zero.
\end{proof}

We now note the following:

\begin{lemma} \label{lemma:table} Suppose that~$\g$ is a reductive Lie algebra
with a faithful irreducible representation of dimension~$d$ for~$d  \in \{17,26,18\}$.
Then~$\g = \h$ or~$\h \oplus \t$, where~$\h$ is semi-simple and~$\t$ is a rank one torus.
Furthermore, assuming that~$\h$ is simple when ~$d= 18$, then~$\h$ is one of the following:
\begin{center}

\begin{tabular}{|c|c|c|}
\hline
$\h$ & $d$ & $\rank(\h)$  \\
\hline
$\sl_2$ & $17$ & $1$   \\
$\so_{17}$ & $17$ & $8$ \\
$\sl_{17}$ & $17$ & $16$   \\
\hline
$\sl_2$ & $26$ & $1$    \\
$\f_4$ & $26$ &  $4$    \\
$\so_{13} \times \sl_{2}$ & $26$ & $7$  \\
$\sl_{13} \times \sl_{2}$ & $26$ & $13$  \\
$\sp_{26} $ & $26$ & $13$   \\
$\so_{26} $ & $26$ & $13$   \\
$\sl_{26}$ & $26$ & $25$    \\
\hline
$\sl_2$ & $18$ & $1$   \\
$\sp_{18}$ & $18$ & $9$ \\
$\so_{18}$ & $18$ & $9$ \\
$\sl_{18}$ & $18$ & $17$ \\
\hline
\end{tabular}
\end{center}
\end{lemma}

\begin{proof} It suffices to classify all small (of dimension at most~$27$) 
representations of the simple Lie groups; these may
be computed using the Weyl character formula.
\end{proof}

Let us now return to the possible
cases in which our~$27$ dimensional
Galois representation is reducible,
and consider the corresponding monodromy groups.
 Suppose there is a constituent of dimension~$17$.
The rank must be bounded by~$6$.
From Lemma~\ref{lemma:table}, it follows that the Lie algebra
of the monodromy group on this summand  must be~$\sl_2$ or~$\sl_2 \times \t$. But the
rank of these algebras is at most~$2$,
 which violates Corollary~\ref{corr:quote}.
 Suppose there is a constituent of dimension~$26$. Then by Corollary~\ref{corr:quote},
  the rank of the Lie algebra of this representation is exactly~$6$, and hence the rank of~$\h$
 with~$\g = \h$ or~$\h \times \t$ is~$6$ or~$5$. Since there are no such groups 
of this rank with irreducible representations of dimension~$26$ by Lemma~\ref{lemma:table},
we once more derive a contradiction.
Hence the only remaining possibility is that the~$27$ dimensional
representation decomposes into two irreducible pieces of dimensions~$9$ and~$18$, corresponding
to a decomposition of weights~$\Sigma = \Sigma_{9} \cup \Sigma_{18}$.
On the other hand, we know that the~$18$ dimensional representation
must have a unipotent element with Jordan blocks of size~$1 + 17$.
The tensor product of two Jordan blocks of size~$m$ and~$n$  with~$n \le m$ decomposes into blocks
of size~$n+m-1$, $n + m -3, \ldots, n -m + 1$. In particular, it must be the case that the~$18$-dimensional
representation does not factor into a product
of smaller dimensional representations, since otherwise there could
not be a Jordan block of a unipotent element of size as larger as~$17$.
Hence the monodromy group on this representation
must have a simple Lie algebra (up to a torus).
Again, we deduce from Lemma~\ref{lemma:table} and using rank considerations that~$\g = \sl_2$ or~$\sl_2 \times \t$, once
more contradicting Corollary~\ref{corr:quote}.

We conclude:
\begin{corr}
Let $\{r_{\lambda} \colon \Gammma_F \to \mr{GL}_{27}(\overline{M}_{\lambda})\}_{\lambda}$ be the compatible system produced in Theorem \ref{theorem:lifts}. Then for all $\lambda$, the Zariski closure of the image of $r_{\lambda}$ is isomorphic to $G$.
\end{corr}
 
 \section{\texorpdfstring{$E_6$}{E6} motives over CM fields} \label{section:motives}
We conclude by showing that $r_{\lambda}$ is a sub-representation of the cohomology of some smooth projective variety over $F$, thus completing the proof of Theorem \ref{theorem:main}. 
Note that, according to the Tate conjecture, we expect that such a compatible family
should be cut out by correspondences over~$F$, and thus arise from a motive~$M$ over~$F$.
 We do not have any
idea how to prove this. On the other hand, we do know that the compatible family~$r_{\lambda}$
becomes automorphic over a CM extension~$H/F/\Q$, and (since we are in highly regular weight),
using standard methods combined
with the work of Shin~\cite{MR2800722}, one can associate
a motive~$M$ over~$H$ whose associated~$p$-adic Galois representations~$\{r_{\lambda} |_{G_H}\}$
have monodromy group~$E_6$.
More precisely, we should say that one \emph{expects} to be able to associate such a motive
where the correspondences cutting out~$M$ arise from Hecke operators. In practice, 
we take a shortcut and 
deduce from~\cite{MR2800722} the weaker claim that the Galois representations over~$G_{H}$
(and thus over~$G_F$ by restriction of scalars) came from cohomology. We apologize
for the omission  and leave it as an exercise to the more responsible reader to write down the correct argument.

To set up all the required notation would be quite cumbersome, so we will simply use the notation of \cite{MR2800722}, giving precise references to where the relevant terms are defined. We hope that a reader with a copy of \cite{MR2800722} at hand can easily follow this argument. Fix an isomorphism $\iota_{\ell} \colon \Qbar_{\elll} \xrightarrow{\sim} \C$; it is implicit in all of the constructions of \cite{MR2800722}. By Theorem \ref{theorem:lifts}, there is a CM extension $H/F$ and a cuspidal automorphic representation $\Pi^0$ of $\mr{GL}_{27}(\mathbf{A}_H)$ such that
\begin{itemize}
\item $(\Pi^0)^\vee \cong (\Pi^0)^c$.
\item $R_{\elll}(\Pi^0) \cong r_{\lambda}|_{\Gammma_H}$, in the notation of \cite[Theorem 7.5]{MR2800722}.
\item $[H^+:\Q] \geq 2$, and $H$ contains a quadratic imaginary field (we can simply enlarge an initial choice of $H$ to ensure these conditions).
\end{itemize}
Set $n=27$, for ease of reference to \cite{MR2800722}; we will recall the construction of $R_{\elll}(\Pi^0)$ and see as a result that after some further base-change that there is an explicit description of this Galois representation in the cohomology of a unitary similitude group Shimura variety. We begin with two reductions. Let $E$ be an imaginary quadratic field not contained in $H$ satisfying the four bulleted conditions in Step (II) of the proof of \cite[Theorem 7.5]{MR2800722}. Replace $H$ by $HE$ and $\Pi^0$ by $\mr{BC}_{HE/H}(\Pi^0)$. Then \textit{having made this replacement} the triple $(E, H, \Pi^0)$ satisfies the six bulleted conditions at the beginning of Step (I) of the proof of \cite[Theorem 7.5]{MR2800722}. Let $H'$ be an imaginary quadratic extension of $H^+$ satisfying the three bulleted conditions (defining the set denoted $\mathcal{F}(H)$ --- but note our $H$ is Shin's $F$) in Step (I) of \cite[Theorem 7.5]{MR2800722}. Then replace $H$ by $HH'$ and $\Pi^0$ by $\mr{BC}_{HH'/H}(\Pi^0)$. Again having made this replacement, Proposition 7.4 of \cite{MR2800722} now applies to the triple $(E, H, \Pi^0)$. There is a Hecke character $\psi$ of $\mathbf{A}_E^\times/E^\times$ such that, setting $\Pi= \psi \otimes \Pi^0$ (an automorphic representation of the group $\mathbb{G}_n(\mathbf{A}) \cong \mr{GL}_1(\mathbf{A}_E) \times \mr{GL}_n(\mathbf{A}_H)$ of \cite[\S 3.1]{MR2800722}), we have (in the notation of \cite[Corollary 6.8]{MR2800722})
\[
R_{\elll}(\Pi^0):= R_{\elll}'(\Pi):= \widetilde{R}_{\elll}'(\Pi) \otimes \mr{rec}_{\elll, \iota_\elll}(\psi^c)|_{\Gammma_H},
\]
where (see \cite[5.5, 6.23]{MR2800722}, and note that the group $G$ no longer denotes $E_6$, but rather the unitary similitude group defined in \cite[\S 5.1]{MR2800722}!)
\[
C_G \cdot \widetilde{R}'_{\elll}(\Pi)= \sum_{\pi^{\infty} \in \mathcal{R}_{\elll}(\Pi)} R^{n-1}_{\xi, \elll}(\pi^\infty)^{\mr{ss}}.
\]
In our case, the integer $C_G= \tau(G) \ker^1(\Q, G)$ (defined in \cite[Theorem 6.1]{MR2800722}) is 2: this is explained in \cite[p. 411--412]{MR2966704}. Moreover, $\widetilde{R}'_{\elll}(\Pi)$ is irreducible (since the image of $r_{\lambda}|_{\Gamma_H}$ is Zariski-dense in $E_6$), and $\mathcal{R}_{\elll}(\Pi)$ in fact contains at least two elements: following \cite[p. 413]{MR2966704}, there will be two automorphic representations of $G(\mathbf{A}_\Q)$, differing by a twist but having isomorphic base-changes to $\mr{GL}_n(\mathbf{A}_H)$, that contribute to $\mathcal{R}_{\elll}(\Pi)$ (these are denoted $\wpi$ and~$\wpi \otimes (\delta_{A/\Q} \circ \nu)$ in \cite[p. 413]{MR2966704}). We fix one such $\wpi^{\infty} \in \mathcal{R}_{\elll}(\Pi)$. It follows that $\widetilde{R}'_{\elll}(\Pi) \cong R^{n-1}_{\xi, \elll}(\wpi^\infty)$ (no semisimplification necessary because these are irreducible representations).

Now recall the decomposition (\cite[5.5]{MR2800722})
\[
H^{n-1}(\mr{Sh}, \mathcal{L}_{\xi})= \bigoplus_{\pi^\infty} \pi^{\infty} \otimes R^{n-1}_{\xi, \elll}(\pi^\infty).
\]
At some finite level $U$, we deduce that $\widetilde{R}'_{\elll}(\Pi)$ is contained in $H^{n-1}(\mr{Sh}_U \times_H \overline{H}, \mathcal{L}_{\xi})$ (it is even a direct summand cut out by Hecke operators). Finally, letting $\mathcal{A}_U \to \mr{Sh}_U$ denote the universal abelian scheme (arising from the PEL moduli problem), and letting $\mathcal{A}_U^{(m)}$ denote its $m$-fold fiber product over $\mr{Sh}_U$ (for any integer $m \geq 1$), then there are integers $m_{\xi}$ and $t_{\xi}$ (see \cite[p. 98]{harris-taylor}) such that
\[
H^{n-1}(\mr{Sh}_U \times_H \overline{H}; \mathcal{L}_{\xi}) \cong \varepsilon \cdot H^{n-1+m_{\xi}}(\mathcal{A}_U^{(m_\xi)} \times_H \overline{H}, \Qbar_{\elll})(t_\xi),
\]
where $\varepsilon$ is a suitable idempotent projector. We recall that our assumption~$[H^+:\Q] \geq 2$ implies that $\mr{Sh}_U$, and therefore $\mathcal{A}_U^{(m_\xi)}$, are smooth projective varieties over~$H$. Thus $r_{\lambda}|_{\Gammma_H} \otimes \mr{rec}_{\elll, \iota_\elll}(\psi^c)^{-1}|_{\Gammma_H}$ is a sub-representation of the cohomology of the smooth projective variety $\mathcal{A}_U^{(m_\xi)}$ over~$H$. Possibly replacing $H$ by a finite extension, we can find a product of CM abelian varieties $A/H$ such that $\mr{rec}_{\elll, \iota_\elll}(\psi^c)|_{\Gammma_H}$ is a sub-representation of $H^{i}(A \times_H \overline{H}, \Qbar_{\elll})(j)$ for some integers $i$ and $j$ (\cite[IV. Proposition D.1]{DMOS}). We conclude that $r_{\lambda}|_{\Gammma_H}$ is a sub-representation of $H^r(X \times_H \overline{H}, \Qbar_{\elll})(s)$ for some smooth projective variety $X/H$ (namely, $X= A \times_H \mathcal{A}_U^{(m_\xi)}$). By Frobenius reciprocity and irreducibility of $r_{\lambda}$, $r_{\lambda}$ (as $\Gammma_F$-representation) is a sub-representation of $\mr{Ind}_{\Gammma_H}^{\Gammma_F} \left(H^{r}(X \times_H \overline{H}, \Qbar_{\elll})(s) \right)$. This induction is just the cohomology of $X$ regarded as a variety over $F$ (i.e., via $X \to \mr{Spec}(H) \to \mr{Spec}(F)$), so the proof of Theorem \ref{theorem:main} is complete. 

\begin{remark} \label{remark:irregular} 
One may reasonably ask whether there exist~$E_6$ motives (or strongly compatible systems) over~$\Q$. We do not know the answer.
There is, however, a technical obstruction for applying the methods of this paper. To use automorphic methods,
 the Hodge--Tate weights of~$r_{\mathrm{min}} \circ \rho_{\lambda}$  must be distinct. However, there cannot exist such a compatible
system over~$\Q$ (or any totally real field), since, by Corollary~5.4.3 of~\cite{BLGGT}, this would imply that the trace of complex conjugation
representation must be~$\pm 1$, and there are no such involutions in~$E_6(\C) \subset \GL_{27}(\C)$.
\end{remark}

 \section{Complements}
 
 We note the following application of Theorem~\ref{theorem:main} to the inverse Galois problem.
Let us write~$E^{\sc}_{6}(\F_{\elll})$ for the~$\F_{\elll}$-points of the (simply connected) form 
of~$E_6$ that we have been considering. The group~$E^{\sc}_{6}(\F_{\elll})$ is the Schur
cover of the simple Chevalley group of type~$E_6$ (with Schur multiplier of order~$(3,\elll-1)$),
which we denote below by~$E_6(\F_{\elll})$.
The groups~$E^{\sc}_{6}(\F_{\elll})$ and~$E^{\sc}_{6}(\F_{\elll}).2$ are known to occur
as Galois groups over~$\Q$ for~$p \equiv 4,5,6,9,16,17 \mod 19$
(these are primes of order~$9$ in~$\F^{\times}_{19}$) by~\cite[Thm~2.3]{Malle}
as a consequence of the rigidity method. In contrast, we can prove 
that~$E^{\sc}_{6}(\F_{\elll}).2$ is a Galois group over~$\Q$ for a positive density of 
primes~$\equiv 1 \mod 19$ (by taking~$E = \Q(\zeta_{19})$ below).
 
 \begin{corr} \label{cor}
   Let~$F$ be an imaginary quadratic field.
  Then, for a set~$S$ of primes~$\elll$ of positive density, there exists a number field~$L/\Q$ containing~$F$ with
  $\Gal(L/\Q) = {}^{L}G(\F_{\elll}) = E^{\sc}_{6}(\F_{\elll}).2$ and~$\Gal(L/F) = E^{\sc}_{6}(\F_{\elll})$.
  Moreover, one may assume that all the primes in~$S$ split completely in any finite extension~$E/\Q$.
  \end{corr}
  
  \begin{proof} We combine the previous result
  with Theorem~3.17 of~\cite{Lars}, to deduce that~$E^{\sc}_6(\F_{{\elll}})$ is the Galois group of the kernel of~$\rho_{\lambda}$ over~$F$
  for a relative density one set of primes~${\elll}$ which split completely in the coefficient field~$M$
of
  the compatible system. In particular, the intersection of~$S$ with the set of 
  primes  split completely in~$E/\Q$
 has positive density. Since the representation~$\rbar_{\lambda}:G_F\to\GL_{27}(\F_{{\elll}})$  is, by construction,
  conjugate self-dual, it extends
  to a representation of~$G_{\Q}$ to $\mathcal{G}_{27}(\mathbf{F}_{\elll})$ whose  kernel therefore has Galois group~${}^{L}G(\F_{\elll})$. 
  \end{proof}
  
 \begin{remark} Note that Corollary~\ref{cor} remains true if one replaces~$E^{\sc}_{6}(\F_{\elll}).2$
 and~$E^{\sc}_{6}(\F_{\elll})$ by~$E_{6}(\F_{\elll}).2$
 and~$E_{6}(\F_{\elll})$, for the obvious reason that the latter groups are quotients of the former groups.
 \end{remark}
 
 \subsection{Acknowledgments}
 This work traces its origins to a discussion following
 a  number theory seminar 
 given by one of us (S.P.) at the University of Chicago in February 2016, where all the authors of this paper were in attendance. 
 We would like to thank Sug Woo Shin for answering questions  related to his paper~\cite{MR2800722}
 relevant for the discussion in~\S\ref{section:motives}.

\bibliographystyle{amsalpha}
\bibliography{E6}

\providecommand{\bysame}{\leavevmode\hbox to3em{\hrulefill}\thinspace}
\providecommand{\MR}{\relax\ifhmode\unskip\space\fi MR }
\providecommand{\MRhref}[2]{%
  \href{http://www.ams.org/mathscinet-getitem?mr=#1}{#2}
}
\providecommand{\href}[2]{#2}
\begin{thebibliography}{BLGGT14}

\bibitem[AH16]{adams-he}
Jeffrey Adams and Xuhua He, \emph{Lifting of elements of {W}eyl groups}, 2016.

\bibitem[BLGGT14]{BLGGT}
Thomas Barnet-Lamb, Toby Gee, David Geraghty, and Richard Taylor,
  \emph{Potential automorphy and change of weight}, Ann. of Math. (2)
  \textbf{179} (2014), no.~2, 501--609. \MR{3152941}

\bibitem[Car12]{Caraiani}
Ana Caraiani, \emph{Local-global compatibility and the action of monodromy on
  nearby cycles}, Duke Math. J. \textbf{161} (2012), no.~12, 2311--2413.
  \MR{2972460}

\bibitem[CHT08]{CHT}
Laurent Clozel, Michael Harris, and Richard Taylor, \emph{Automorphy for some
  {$l$}-adic lifts of automorphic mod {$l$} {G}alois representations}, Publ.
  Math. Inst. Hautes \'Etudes Sci. (2008), no.~108, 1--181, With Appendix A,
  summarizing unpublished work of Russ Mann, and Appendix B by Marie-France
  Vign{\'e}ras. \MR{2470687 (2010j:11082)}

\bibitem[DMOS82]{DMOS}
Pierre Deligne, James~S. Milne, Arthur Ogus, and Kuang-yen Shih, \emph{Hodge
  cycles, motives, and {S}himura varieties}, Lecture Notes in Mathematics, vol.
  900, Springer-Verlag, Berlin, 1982. \MR{654325 (84m:14046)}

\bibitem[DR10]{DR}
Michael Dettweiler and Stefan Reiter, \emph{Rigid local systems and motives of
  type {$G_2$}}, Compos. Math. \textbf{146} (2010), no.~4, 929--963, With an
  appendix by Michael Dettweiler and Nicholas M. Katz. \MR{2660679}

\bibitem[FM95]{FM}
Jean-Marc Fontaine and Barry Mazur, \emph{Geometric {G}alois representations},
  Elliptic curves, modular forms, \& {F}ermat's last theorem ({H}ong {K}ong,
  1993), Ser. Number Theory, I, Int. Press, Cambridge, MA, 1995, pp.~41--78.
  \MR{1363495}

\bibitem[Gro00]{gross}
Benedict~H. Gross, \emph{On minuscule representations and the principal {${\rm
  SL}_2$}}, Represent. Theory \textbf{4} (2000), 225--244 (electronic).
  \MR{1795753}

\bibitem[HT01]{harris-taylor}
Michael Harris and Richard Taylor, \emph{The geometry and cohomology of some
  simple {S}himura varieties}, Annals of Mathematics Studies, vol. 151,
  Princeton University Press, Princeton, NJ, 2001, With an appendix by Vladimir
  G. Berkovich. \MR{MR1876802 (2002m:11050)}

\bibitem[Kis08]{Kisin}
Mark Kisin, \emph{Potentially semi-stable deformation rings}, J. Amer. Math.
  Soc. \textbf{21} (2008), no.~2, 513--546. \MR{2373358}

\bibitem[KW09a]{KWI}
Chandrashekhar Khare and Jean-Pierre Wintenberger, \emph{Serre's modularity
  conjecture. {I}}, Invent. Math. \textbf{178} (2009), no.~3, 485--504.
  \MR{2551763}

\bibitem[KW09b]{KWII}
\bysame, \emph{Serre's modularity conjecture. {II}}, Invent. Math. \textbf{178}
  (2009), no.~3, 505--586. \MR{2551764}

\bibitem[Lar95]{Lars}
M.~Larsen, \emph{Maximality of {G}alois actions for compatible systems}, Duke
  Math. J. \textbf{80} (1995), no.~3, 601--630. \MR{1370110}

\bibitem[LP92]{LP}
M.~Larsen and R.~Pink, \emph{On {$l$}-independence of algebraic monodromy
  groups in compatible systems of representations}, Invent. Math. \textbf{107}
  (1992), no.~3, 603--636. \MR{1150604 (93h:22031)}

\bibitem[Lur01]{Lurie}
Jacob Lurie, \emph{On simply laced {L}ie algebras and their minuscule
  representations}, Comment. Math. Helv. \textbf{76} (2001), no.~3, 515--575.
  \MR{1854697}

\bibitem[Mal88]{Malle}
Gunter Malle, \emph{Exceptional groups of {L}ie type as {G}alois groups}, J.
  Reine Angew. Math. \textbf{392} (1988), 70--109. \MR{965058}

\bibitem[NSW00]{Neukirch}
J{\"u}rgen Neukirch, Alexander Schmidt, and Kay Wingberg, \emph{Cohomology of
  number fields}, Grundlehren der Mathematischen Wissenschaften [Fundamental
  Principles of Mathematical Sciences], vol. 323, Springer-Verlag, Berlin,
  2000. \MR{MR1737196 (2000j:11168)}

\bibitem[Pat16]{Patrikis}
Stefan Patrikis, \emph{Deformations of {G}alois representations and exceptional
  monodromy}, Inventiones mathematicae \textbf{205} (2016), no.~2, 269--336.

\bibitem[Pat17]{Patrikis2}
\bysame, \emph{Deformations of {G}alois representations and exceptional
  monodromy, {II}: raising the level}, Math. Ann. \textbf{368} (2017), no.~3-4,
  1465--1491. \MR{3673661}

\bibitem[Ram02]{Ramakrishna}
Ravi Ramakrishna, \emph{Deforming {G}alois representations and the conjectures
  of {S}erre and {F}ontaine-{M}azur}, Ann. of Math. (2) \textbf{156} (2002),
  no.~1, 115--154. \MR{1935843}

\bibitem[Rob17]{Dave}
David Roberts, \emph{Newforms with rational coefficients}, to appear in the
  Ramanujan Journal, 2017.

\bibitem[Sch90]{MR1047142}
A.~J. Scholl, \emph{Motives for modular forms}, Invent. Math. \textbf{100}
  (1990), no.~2, 419--430. \MR{1047142}

\bibitem[Ser94]{SerreMotives}
Jean-Pierre Serre, \emph{Propri\'et\'es conjecturales des groupes de {G}alois
  motiviques et des repr\'esentations {$l$}-adiques}, Motives ({S}eattle, {WA},
  1991), Proc. Sympos. Pure Math. \textbf{55} (1994), 377--400.

\bibitem[Ser08]{SerreNotes}
\bysame, \emph{Topics in {G}alois theory}, second ed., Research Notes in
  Mathematics, vol.~1, A K Peters, Ltd., Wellesley, MA, 2008, With notes by
  Henri Darmon. \MR{2363329}

\bibitem[Shi11]{MR2800722}
Sug~Woo Shin, \emph{Galois representations arising from some compact {S}himura
  varieties}, Ann. of Math. (2) \textbf{173} (2011), no.~3, 1645--1741.
  \MR{2800722}

\bibitem[Tay12]{MR2966704}
Richard Taylor, \emph{The image of complex conjugation in {$l$}-adic
  representations associated to automorphic forms}, Algebra Number Theory
  \textbf{6} (2012), no.~3, 405--435. \MR{2966704}

\bibitem[TY07]{Yoshida}
Richard Taylor and Teruyoshi Yoshida, \emph{Compatibility of local and global
  {L}anglands correspondences}, J. Amer. Math. Soc. \textbf{20} (2007), no.~2,
  467--493. \MR{2276777}

\bibitem[Yun14]{Yun}
Zhiwei Yun, \emph{Motives with exceptional {G}alois groups and the inverse
  {G}alois problem}, Invent. Math. \textbf{196} (2014), no.~2, 267--337.
  \MR{3193750}

\end{thebibliography}

\end{document}